\documentclass{article}

\usepackage[english]{babel}
\usepackage[utf8]{inputenc}
\usepackage[T1]{fontenc}

\usepackage[a4paper,top=3cm,bottom=2cm,left=3cm,right=3cm,marginparwidth=2.0cm]{geometry}
\usepackage{amssymb, amsmath, amsthm}

\usepackage{amsthm}
\usepackage{amssymb}
\usepackage{amsmath}
\usepackage{mathtools}
\usepackage{epsfig}
\usepackage{caption}
\usepackage[colorinlistoftodos]{todonotes}
\usepackage[colorlinks=true, allcolors=blue]{hyperref}
\usepackage{xcolor}
\usepackage{pinlabel}
\usepackage{geometry}
\usepackage{subcaption}
\usepackage[numbers]{natbib} 
\usepackage{hyperref}
\usepackage{authblk}
\usepackage{longtable}
\usepackage[shortlabels]{enumitem}
\usepackage{multicol}
\usepackage{subcaption} 
\usepackage{authblk}
\usepackage{arydshln} 

\theoremstyle{definition}
\newtheorem{definition}{Definition}[section]

\theoremstyle{plain}
\newtheorem{theorem}{Theorem}[section]
\newtheorem{lemma}{Lemma}[section]
\newtheorem{proposition}{Proposition}[section]

\theoremstyle{remark}

\newtheorem{corollary}{Corollary}

\newcommand\id{\mathrm{Id}}

\title{Bonded braids and the Markov theorem}

\author[1]{Paolo Cavicchioli}
\author[2]{Boštjan Gabrovšek}
\author[3]{Matic Simonič}

\affil[1]{Institute of Mathematics, Physics and Mechanics, Jadranska ulica 19,1000 Ljubljana, Slovenia. Email adress, P. Cavicchioli: }
\affil[2]{ Rudolfovo, Podbreznik 15, 8000 Novo mesto, Slovenia; University of Ljubljana, Faculty
of Education, Kardeljeva ploščad 16, 1000 Ljubljana, Slovenia.
Email address, B. Gabrovšek: bostjan.gabrovsek@pef.uni-lj.si} 
\affil[3]{University of Ljubljana, Faculty of Mathematics and Physics, Jadranska ulica 19, 1000
Ljubljana, Slovenia.
Email address, M. Simonič: matic.simonic@fmf.uni-lj.si
}

\date{July 2025}

\begin{document}

\maketitle

\begin{abstract}
   Bonded knots arise naturally in topological protein modeling, where intramolecular interactions such as disulfide bridges stabilize folded configurations. These structures extend classical knot theory and have been formalized as bonded knots. In this paper, we develop the algebraic theory of bonded braids, introducing the bonded braid monoid in the topological and rigid settings, which encodes both classical crossings and (rigid) bonded connections. We prove bonded analogues of the Alexander and Markov theorems, establishing that every bonded knot arises as the closure of a bonded braid and that two bonded knots are equivalent if and only if their braid representatives are related by a finite sequence of algebraic (Markov-like) moves. In addition, we define the bonded Burau and reduced bonded Burau representations of the monoid, extending classical braid group representations to the bonded setting, and analyze their (non-)faithfulness in low and high dimensions.
\end{abstract}

\section{Introduction}
Proteins can fold into entangled conformations, where the spatial arrangement of the backbone chain forms a nontrivial mathematical knot \cite{Mansfield1994, virnau2006intricate, dabrowski2019knotprot, dabrowski2017tie}.
Intramolecular bonds in proteins, such as disulfide bridges, give rise to complex entanglements that cannot be captured by classical knot theory. To model such structures, recent work has extended knot theory to the theory of bonded knots \cite{adams2020knot, gabrovvsek2021invariant, gabrovsek2025bracket, Goundaroulis2017, gugumcu2022invariants, dabrowski2024theta, bruno2024knots, sulkowska2020folding}. In these models, bonds are represented as embedded arcs connecting points on a knot (or open curve), allowing one to topologically describe folding and entanglement phenomena in biological macromolecules (see \autoref{fig:valvet}).

\begin{figure}[h]
    \centering
\includegraphics[width=0.75\textwidth, page=1]{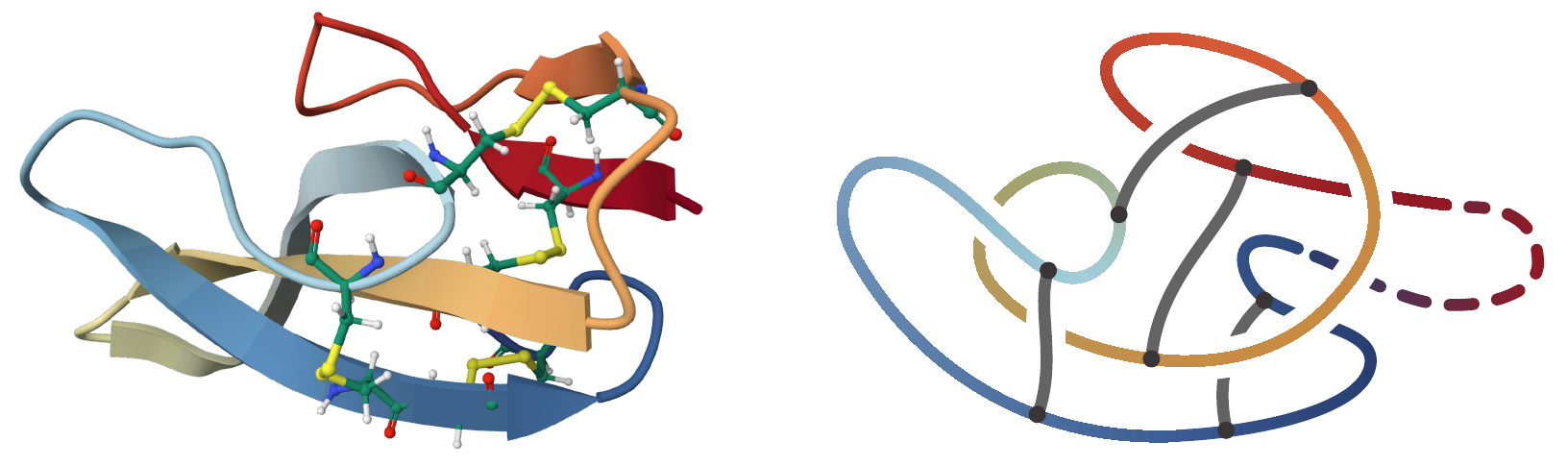}
    \caption{Left: 3D ribbon model of the FS2 toxin peptide (pdb 1tfs). Right: corresponding bonded knot diagram with a direct closure. The multiple disulfide bonds stabilize the compact, active conformation of the toxin. }
    \label{fig:valvet}
\end{figure}

\subsection{Bonded knots}

A (topological) \emph{bonded knot} is a pair $(L, \mathbf{b})$, where
$K$ is an oriented knot embedded in $S^3$ (or $\mathbb{R}^3$)
and $\mathbf{b} = \{b_1, b_2, \ldots, b_n\}$ is a collection of pairwise disjoint \emph{bonds} (closed intervals) embedded in $S^3$, each with endpoints on $K$, i.e., $K \cap b_i = \partial b_i$.
The endpoints of the bonds define trivalent vertices on the knot, and bonded knots may thus be viewed as edge-colored spatial graphs.

\noindent \emph{Remark.} For simplicity, we refer to individual knotted components as ``knots'' throughout the paper, but all constructions and theorems hold without modification in the setting of links.

A \emph{bonded knot diagram} is a regular planar projection of the spatial graph, with transverse double points (crossings) and over/under information included (see \autoref{fig:valvet}).

Two bonded knots $K_1$ and $K_2$ are said to be ambient isotopic if there exists a smooth isotopy $H: S^3 \times [0,1] \to S^3$ such that $H_0 = \text{id}_{S^3}$ and $H_1(K_1) = K_2$. This isotopy is generated by Reidemeister moves I–V as the following theorem states.

\begin{theorem}[\cite{gabrovvsek2021invariant, Kauffman1989}]
Two bonded knots are ambient isotopic if and only if their diagrams are related by a finite sequence of Reidemeister moves I--V depicted in \autoref{fig:reid}.
\end{theorem}

\begin{figure}[ht]
    \centering
    \begin{subfigure}[b]{0.3\textwidth}\centering
        \includegraphics[scale=1.0, page=1]{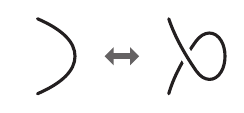}
        \caption{move I}
    \end{subfigure}
    \begin{subfigure}[b]{0.3\textwidth}\centering
        \includegraphics[scale=1.0, page=2]{reid.pdf}
        \caption{move II}
    \end{subfigure}
        \begin{subfigure}[b]{0.3\textwidth}\centering
        \includegraphics[scale=1.0, page=3]{reid.pdf}
        \caption{move III}
    \end{subfigure}

    \vspace{0.5cm} 
    
    \begin{subfigure}[b]{0.3\textwidth}\centering
        \includegraphics[scale=1.0, page=4]{reid.pdf}
        \caption{move IV}
    \end{subfigure}
    \begin{subfigure}[b]{0.3\textwidth}\centering
        \includegraphics[scale=1.0, page=5]{reid.pdf}
        \caption{move IV'}
    \end{subfigure}
        \begin{subfigure}[b]{0.3\textwidth}\centering
        \includegraphics[scale=1.0, page=6]{reid.pdf}
        \caption{move V}
    \end{subfigure}
    \caption{Reidemeister moves for bonded knots. Any part of the depicted arcs can be either a knot arc or a band arc.}
 \label{fig:reid}
\end{figure}

If we restring the twisting at vertices, we obtain a class of rigid \emph{rigid-vertex bonded knots}, where move V is replaced by the rigid move $V_{R}$ (or $V_{R}'$) depicted in \autoref{fig:rigid}, see \cite{gabrovvsek2021invariant, Kauffman1989}.

To distinguish explicitly between the rigid and non-rigid settings, we refer to non-rigid-vertex bonded knots as \emph{topological} bonded knots.

\begin{definition}
Two rigid-vertex bonded knots are rigid-vertex isotopic if their diagrams are related by a finite sequence of Reidemeister moves I--IV and $\mathrm{V_R}$.
\end{definition}

\begin{figure}[ht]
\centering
\begin{tabular}{ccc}
\includegraphics[scale=1.0,page=7]{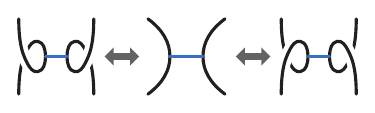} & \raisebox{2em}{or} & \includegraphics[scale=1.0,page=8]{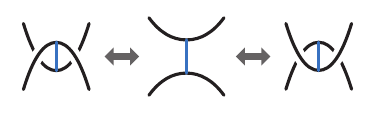} \\
(a) move $\mathrm{V_R}$ & & (b) move $\mathrm{V_R'}$
\end{tabular}
\caption{Rigid versions of the move V. It was shown in \cite{gabrovvsek2021invariant} that the move $\mathrm{V_R}$ is a consiquence of the move $\mathrm{V_R'}$.}
\label{fig:rigid}
\end{figure}

In order to simplify computations, we often consider bonded knot diagrams, where the bonds are locally trivial, i.e. they are contained inside an arbitrary small 2-disk without crossings. Such diagrams have been studied in various contexts, including \cite{adams2020knot, gabrovvsek2021invariant, gabrovsek2025bracket, Goundaroulis2017, dabrowski2018aps}. If a bonded knot diagram does not have crossings on its bonds, we call such a diagram \emph{a diagram with isolated bonds}. A bonded knot can always be represented by a diagram with isolated bonds, since any bond can be contracted via VI-moves to an arbitrarily small bond without crossings, as illustrated in \autoref{fig:isolate_bond}, which shows the following lemma.

\begin{figure}[ht]
    \centering
    \includegraphics[width=0.8\textwidth, page=2]{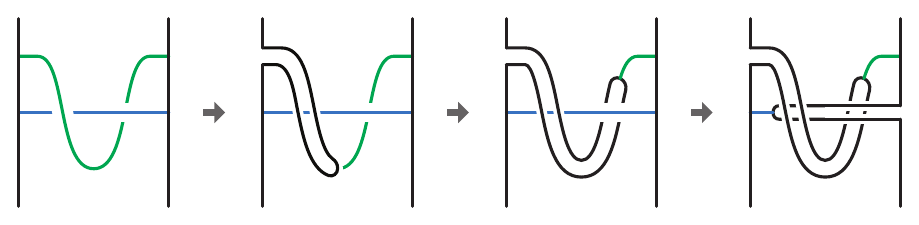}
    \caption{Isolating bonds via VI-moves.}
    \label{fig:isolate_bond}
\end{figure}

\begin{lemma}
    Given a bonded knot diagram $D$ of $K$, there exists a finite sequence of Reidemeister moves that takes $D$ into a diagram with isolated bonds, $D'$. 
\end{lemma}


From now on, we will assume that all bonded knot diagrams have isolated bonds. 
Reidemeister moves in \autoref{fig:reid} may violate bond isolation.
However, there is always a sequence of Reidmesiter moves, or a combination of Reidemeister moves, that keep the bonds isolated. We formalize this by the following lemma.


\begin{lemma}
    Let $D$ and $D'$ be two topological bonded knot diagrams with isolated bonds. The diagrams represent ambient isotopic topological bonded knots if and only if they are connected through a finite sequence of moves $\mathrm{I_B}$, $\mathrm{II_B}$, $\mathrm{III_B}$, $\mathrm{V_B}$, $\mathrm{VI_B}$, and $\mathrm{VI_B'}$ depicted in \autoref{fig:bonded_reid}.
\end{lemma}

\begin{proof} In one direction the proof is trivial. In the other direction, we argue as follows.
    Since $D$ and $D'$ represent ambient isotopic topological bonded knots, they are connected by through a finite sequence of moves I, II, III, IV, and V. Moves $\mathrm{I_B}$, $\mathrm{II_B}$, and $\mathrm{III_B}$ are just restrictions of moves I, II, and III, which do not involve the bond. 
    Note that the move $\mathrm{VI_B}$ is composed of two IV moves.  Since $D $ and $D'$ have isolated bonds we have an even number of IV moves over each bond in the sequence of moves connecting them; hence, IV moves over a bond can be paired into $\mathrm{VI_B}$ moves over the same bond. A similar argument holds for IV' and $\mathrm{VI_B'}$
    We now obtain a finite sequence of moves $\mathrm{I_B}$, $\mathrm{II_B}$, $\mathrm{III_B}$, $\mathrm{V_B}$, $\mathrm{VI_B}, \mathrm{VI_B'}$ that connects $D$ and $D'$.
    
\end{proof}

\begin{figure}[ht]
    \centering
    \begin{subfigure}[b]{0.3\textwidth}\centering
        \includegraphics[scale=1.0, page=1]{reid.pdf}
        \caption{move $\mathrm{I_B}$}
    \end{subfigure}
    \begin{subfigure}[b]{0.3\textwidth}\centering
        \includegraphics[scale=1.0, page=2]{reid.pdf}
        \caption{move $\mathrm{II_B}$}
    \end{subfigure}
        \begin{subfigure}[b]{0.3\textwidth}\centering
        \includegraphics[scale=1.0, page=3]{reid.pdf}
        \caption{move $\mathrm{III_B}$}
    \end{subfigure}

    \vspace{0.5cm} 
    \begin{subfigure}[b]{0.3\textwidth}\centering
        \includegraphics[scale=1.0, page=15]{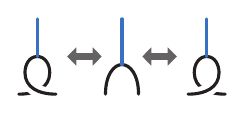}
        \caption{move $\mathrm{V_B}$}
    \end{subfigure}
    \begin{subfigure}[b]{0.3\textwidth}\centering
        \includegraphics[scale=1.0, page=11]{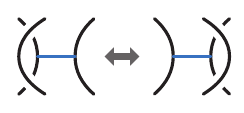}
        \caption{move $\mathrm{VI_B}$}
    \end{subfigure}
    \begin{subfigure}[b]{0.3\textwidth}\centering
        \includegraphics[scale=1.0, page=12]{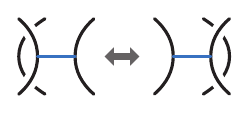}
        \caption{move $\mathrm{VI'_B}$}
    \end{subfigure}

    \caption{Reidemeister moves for knots with isolated bonds. Note that moves I, II, and III cannot be applied to a strand that represents a bond.}
 \label{fig:bonded_reid}
\end{figure}

A similar lemma can be formed and shown by the same means for the rigid case:

\begin{lemma}
    Let $D$ and $D'$ be two rigid bonded knot diagrams with isolated bonds. The diagrams represent ambient isotopic topological bonded knots if and only if they are connected through a finite sequence of moves $\mathrm{I_B}$, $\mathrm{II_B}$, $\mathrm{III_B}$, $\mathrm{V_R}$, $\mathrm{V_R'}$, $\mathrm{VI_R}$, $\mathrm{VI_R'}$.
\end{lemma}


\medskip \emph{Remark.} In literature we can find also the move depicted in the top row of \autoref{fig:isolated_moves2}, but this move is a consequence of the moves $\mathrm{II_B}$ and $\mathrm{VI_B}$, as shown in the bottom row of \autoref{fig:isolated_moves2}.

\begin{figure}[ht]
    \centering
    \begin{subfigure}[b]{0.3\textwidth}
        \centering
        \includegraphics[scale=1.0,page=9]{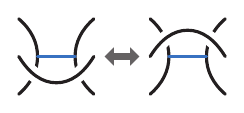}
        \caption{move $\mathrm{VII_B}$}
    \end{subfigure}
    \qquad
    \begin{subfigure}[b]{0.3\textwidth}
        \centering
        \includegraphics[scale=1.0,page=10]{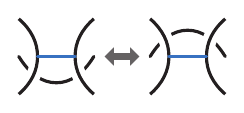}
        \caption{move $\mathrm{VII'_B}$}
    \end{subfigure}

    \vspace{0.5cm}

    \makebox[\textwidth]{%
        \begin{subfigure}[b]{\textwidth}
            \centering
            \includegraphics[scale=1.0,page=14]{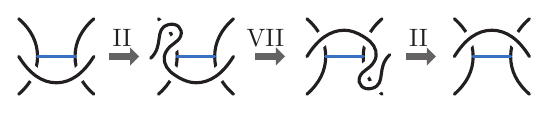}
            \caption{expressing move move $\mathrm{VII_B}$ with moves move $\mathrm{II_B}$ and move $\mathrm{VI_B}$}
        \end{subfigure}
    }

    \caption{Top row: a Reidemeister move for an isolated bond, which is a consequence of moves II and VII, as shown in the bottom row.}
    \label{fig:isolated_moves2}
\end{figure}

\section{The bonded braid monoid and group}
In this section, we introduce an algebraic, braid based description of bonded knots by extending classical braid theory to a theory of bonded braids.

\begin{definition}\label{def:mn}
Let $ n \in \mathbb{N} $. We define $ M_n $, the \emph{topological bonded braid monoid} on $ n $ strands, to be the monoid generated by the following elements (see \autoref{fig:gen} for a visual representation of the generators and \autoref{fig:braid-relations} for a visual representation of the relations):
\begin{itemize}
    \item $ \sigma_1^{\pm1}, \sigma_2^{\pm1}, \dots, \sigma_{n-1}^{\pm1},$ \hfill (standard braid generators and their inverses)
    \item $ b_1, b_2, \dots, b_{n-1} $. \hfill (bond generators)
\end{itemize}

\begin{figure}[htbp]
    \centering
    \begin{subfigure}[b]{0.22\textwidth}
        \centering
        \includegraphics[scale=1.0,page=3]{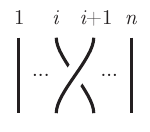}
        \caption{$ \sigma_i $}
    \end{subfigure} \hfill
    \begin{subfigure}[b]{0.22\textwidth}
        \centering
        \includegraphics[scale=1.0,page=4]{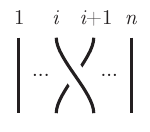}
        \caption{$ \sigma_i^{-1} $}
    \end{subfigure} \hfill
    \begin{subfigure}[b]{0.22\textwidth}
        \centering
        \includegraphics[scale=1.0,page=5]{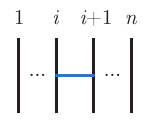}
        \caption{$ b_i $}
    \end{subfigure} \hfill
    \begin{subfigure}[b]{0.22\textwidth}
        \centering
        \includegraphics[scale=1.0,page=33]{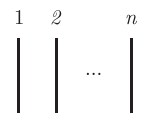}
        \caption{$ \id_n $}
    \end{subfigure}
    \caption{Generators of the topological bonded braid monoid: the braid generator $ \sigma_i $, its inverse $ \sigma_i^{-1} $, the bond generator $ b_i $, and the identity braid $ \id_n $.}
    \label{fig:gen}
\end{figure}

These generators satisfy the following relations:
\begin{itemize}
    \item \textbf{standard braid relations:}
    \begin{align*}
        \sigma_i \sigma_i^{-1} &= \id_n, &&  i = 1, \dots, n-1 \tag{R1} \\
        \sigma_i \sigma_j &= \sigma_j \sigma_i, &&  |i-j| \geq 2 \tag{R2} \\
        \sigma_i \sigma_{i+1} \sigma_i &= \sigma_{i+1} \sigma_i \sigma_{i+1}, &&  i = 1, \dots, n-2 \tag{R3}
    \end{align*}

    \item \textbf{bonded braid relations:}
    \begin{align*}
        b_i b_j &= b_j b_i, &&  |i-j| \geq 2 \tag{B1}
    \end{align*}

    \item \textbf{mixed braid relations:}
    \begin{align*}
        \sigma_i b_j &= b_j \sigma_i, &&  |i-j| \geq 2 \tag{M1} \\
        \sigma_i b_i &= b_i \sigma_i, &&  i = 1, \dots, n-1 \tag{M2} \\
        \sigma_{i+1} \sigma_i b_{i+1} &= b_i \sigma_{i+1} \sigma_i, &&  i = 1, \dots, n-2 \tag{M3} \\
        \sigma_i \sigma_{i+1} b_i &= b_{i+1} \sigma_i \sigma_{i+1}, &&  i = 1, \dots, n-2 \tag{M4}
    \end{align*}
\end{itemize}

A same set of relations also appear in \cite{kauffman1992link} and \cite{birman1993new}.
Note that restricting to the generators $ \sigma_i^{\pm1} $ and the relations \textnormal{(R1)–(R3)} recovers the standard Artin braid group $B_n$, so $B_n$ is a subgroup of $M_n$.
\end{definition}

\begin{figure}[htbp]
    \centering
    \begin{subfigure}[b]{0.4\textwidth}
    \centering
        \includegraphics[scale=1.0,page=7]{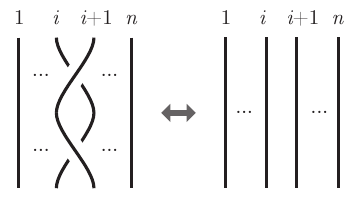}
        \caption{R1}
    \end{subfigure} \qquad
    \begin{subfigure}[b]{0.4\textwidth}
    \centering
        \includegraphics[scale=1.0,page=8]{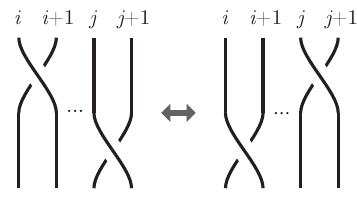}
        \caption{R2}
    \end{subfigure}

    \vspace{0.5cm} 

    \begin{subfigure}[b]{0.4\textwidth}
    \centering
        \includegraphics[scale=1.0,page=10]{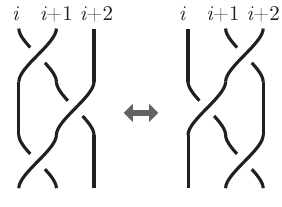}
        \caption{R3}
    \end{subfigure} \qquad
    \begin{subfigure}[b]{0.4\textwidth}
    \centering
        \includegraphics[scale=1.0,page=9]{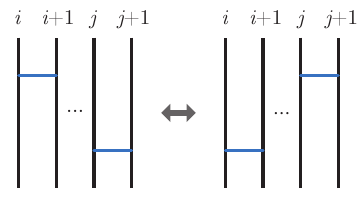}
        \caption{B1}
    \end{subfigure}

    \vspace{0.5cm}

        \begin{subfigure}[b]{0.4\textwidth}
    \centering
        \includegraphics[scale=1.0,page=11]{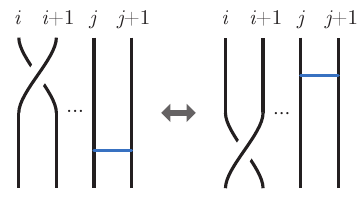}
        \caption{M1}
    \end{subfigure} \qquad
    \begin{subfigure}[b]{0.4\textwidth}
    \centering
        \includegraphics[scale=1.0,page=12]{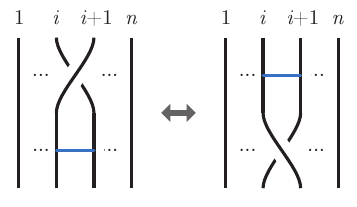}
        \caption{M2}
    \end{subfigure}

    \vspace{0.5cm}

        \begin{subfigure}[b]{0.4\textwidth}
    \centering
        \includegraphics[scale=1.0,page=31]{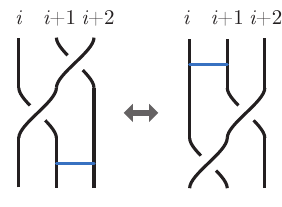}
        \caption{M3}
    \end{subfigure} \qquad
    \begin{subfigure}[b]{0.4\textwidth}
    \centering
        \includegraphics[scale=1.0,page=32]{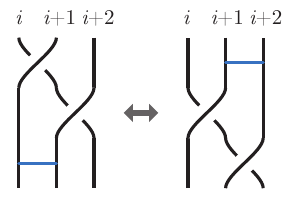}
        \caption{M4}
    \end{subfigure}

    \caption{Visualization of the relations in the bonded braid monomial.}
    \label{fig:braid-relations}
\end{figure}

A definition of a rigid bonded knot needs to violate the condition that all strands are monotonically descending top-to-bottom, but we will show that this condition can be violated only locally (as in \autoref{fig:ki}), which gives motivation of the following definition of the rigid bonded braid monoid.

\begin{definition}
Let $ n \in \mathbb{N} $. Define $ RM_n $, the \emph{rigid bonded braid monoid} on $ n $ strands, to be the monoid generated by the following elements (see \autoref{fig:generators1} for a visual representation of the generators):
\begin{itemize}
    \item $ \sigma_1^{\pm1}, \sigma_2^{\pm1}, \dots, \sigma_{n-1}^{\pm1}$, \hfill (standard braid generators and their inverses)
    \item $ b_1, b_2, \dots, b_{n-1}$, \hfill (bond generators)
    \item $ k_1, k_2, \dots, k_{n-1}$. \hfill (kink generators)
\end{itemize}

The generators satisfy the following relations:
\begin{itemize}
    \item \textbf{standard braid relations:}
    \begin{align*}
        \sigma_i \sigma_i^{-1} &= \id_n, &&  i = 1, \dots, n-1 \tag{R1} \\
        \sigma_i \sigma_j &= \sigma_j \sigma_i, &&  |i-j| \geq 2 \tag{R2} \\
        \sigma_i \sigma_{i+1} \sigma_i &= \sigma_{i+1} \sigma_i \sigma_{i+1}, &&  i = 1, \dots, n-2 \tag{R3}
    \end{align*}

    \item \textbf{bonded braid and kink relations:}
    \begin{align*}
        b_i b_j &= b_j b_i, &&  |i-j| \geq 2 \tag{B1} \\
        k_i k_j &= k_j k_i, &&  |i-j| \geq 2 \tag{K1} \\
        b_i k_j &= k_j b_i, &&  |i-j| \geq 2 \tag{BK1}
    \end{align*}

    \item \textbf{mixed relations:}
    \begin{align*}
        \sigma_i b_j &= b_j \sigma_i, &&  |i-j| \geq 2 \tag{MB1} \\
        \sigma_i b_i &= b_i \sigma_i, &&  i = 1, \dots, n-1 \tag{MB2} \\
        \sigma_{i+1} \sigma_i b_{i+1} &= b_i \sigma_{i+1} \sigma_i, &&  i = 1, \dots, n-2 \tag{MB3} \\
        \sigma_i \sigma_{i+1} b_i &= b_{i+1} \sigma_i \sigma_{i+1}, &&  i = 1, \dots, n-2 \tag{MB4} \\
        \sigma_i k_j &= k_j \sigma_i, &&  |i-j| \geq 2 \tag{MK1} \\
        \sigma_i k_i &= k_i \sigma_i, &&  i = 1, \dots, n-1 \tag{MK2} \\
        \sigma_{i+1} \sigma_i k_{i+1} &= k_i \sigma_{i+1} \sigma_i, &&  i = 1, \dots, n-2 \tag{MK3} \\
        \sigma_i \sigma_{i+1} k_i &= k_{i+1} \sigma_i \sigma_{i+1}, &&  i = 1, \dots, n-2 \tag{MK4}
    \end{align*}
\end{itemize}
\end{definition}

Again, $B_n$ is a subgroup of $M_n$, which is a submonoid of $RM_n$.

\begin{figure}[htbp]
    \centering
    \begin{subfigure}[b]{0.18\textwidth}
        \centering
        \includegraphics[scale=1.0,page=3]{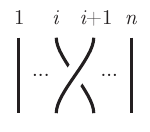}
        \caption{$ \sigma_i $}
    \end{subfigure} \hfill
    \begin{subfigure}[b]{0.18\textwidth}
        \centering
        \includegraphics[scale=1.0,page=4]{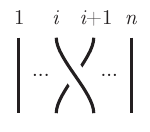}
        \caption{$ \sigma_i^{-1} $}
    \end{subfigure} \hfill
    \begin{subfigure}[b]{0.18\textwidth}
        \centering
        \includegraphics[scale=1.0,page=5]{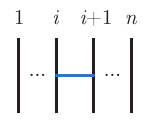}
        \caption{$ b_i $}
    \end{subfigure} \hfill
    \begin{subfigure}[b]{0.18\textwidth}
        \centering
        \includegraphics[scale=1.0,page=22]{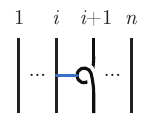}
        \caption{$ k_i $} \label{fig:ki}
    \end{subfigure} 
    \caption{Generators of the rigid bonded braid monoid: the braid generator $ \sigma_i $, its inverse $ \sigma_i^{-1} $, the bond generator $ b_i $, the kink generator $ k_i $.}
    \label{fig:generators1}
\end{figure}

If we complete the monoid $M_n$ by formally adding inverses of elements $b_i$, we obtain the universal group $BB_n$ of $M_n$:

\begin{definition}
Let $ n \in \mathbb{N} $.  The \emph{topological bonded braid group} on $ n $ strands, denoted by $BB_n$, is the group obtained from the topological bonded braid monoid $ M_n $ by adjoining inverses $ b_1^{-1}, b_2^{-1}, \dots, b_{n-1}^{-1} $ to the bond generators $ b_1, b_2, \dots, b_{n-1} $.
\end{definition}

We depict $b_i^{-1}$'s with a dotted line as in \autoref{fig:antigenerators1} and can think of it as an ``anti-bond'' that annihilates a bond.

\begin{definition}
Let $ n \in \mathbb{N} $. The \emph{rigid bonded braid group} on $ n $ strands, denoted by $RB_n$, is the group obtained from the rigid bonded braid monoid $ RM_n $ by adjoining:
\begin{itemize}
    \item inverses $ b_1^{-1}, b_2^{-1}, \dots, b_{n-1}^{-1} $ to the bond generators $ b_1, b_2, \dots, b_{n-1} $, and
    \item inverses $ k_1^{-1}, k_2^{-1}, \dots, k_{n-1}^{-1} $ to the kink generators $ k_1, k_2, \dots, k_{n-1} $ (see \autoref{fig:antigenerators2}).
\end{itemize}
\end{definition}

\begin{figure}[htbp]
    \centering
    \begin{subfigure}[b]{0.25\textwidth}
        \centering
        \includegraphics[scale=1.0,page=6]{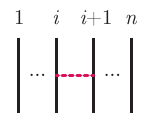}
        \caption{$ b_i^{-1} $}\label{fig:antigenerators1}
    \end{subfigure} \qquad
    \begin{subfigure}[b]{0.25\textwidth}
        \centering
        \includegraphics[scale=1.0,page=34]{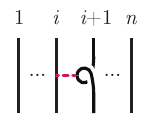}
        \caption{$ k_i^{-1} $}\label{fig:antigenerators2}
    \end{subfigure}
    \caption{Inverse generators in the bonded braid setting: the inverse bond generator $ b_i^{-1} $ (left), and the inverse kink generator $ k_i^{-1} $ (right).}
    \label{fig:antigenerators}
\end{figure}

\begin{definition}
The \emph{closure} $\hat \beta$ of a bonded braid $\beta \in M_n$ (or $RM_n$) is obtained by connecting the corresponding open top and bottom endpoints of the braid, as depicted in \autoref{fig:closure}. 
\end{definition}

\begin{figure}[ht]
    \centering
    \includegraphics[width=0.7\linewidth, page=24]{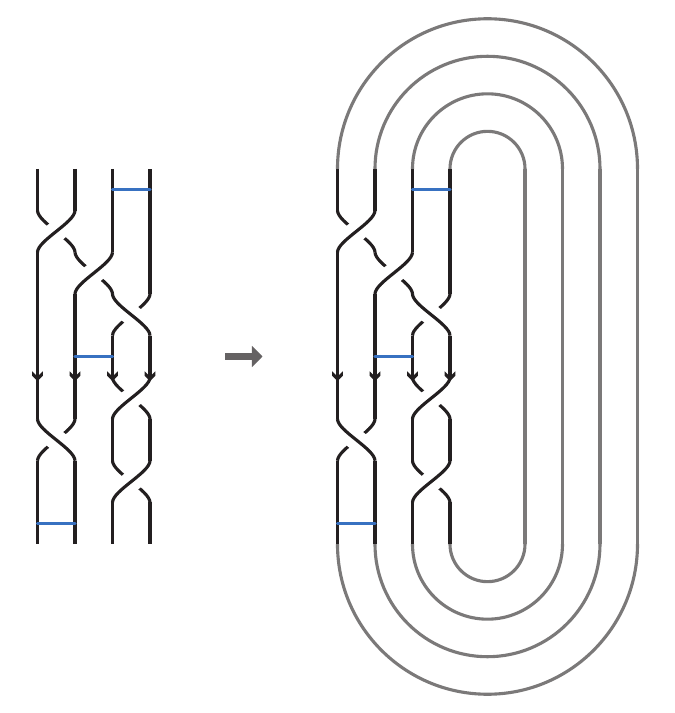}
    \caption{A braid $\beta$ (left) and its closure $\hat \beta$ (right).}
    \label{fig:closure}
\end{figure}

\section{Alexander theorem} \label{sec:alexander}


Before stating the Alexander theorem for bonded knots -- which asserts that every bonded knot can be presented as the closure of a braid -- we note that the most natural setting for its formulation and proof is the piecewise linear (PL) category.

The theory of tame knots is equivalent in both the smooth and piecewise-linear (PL) categories \cite{crowell1977introduction}.
Similarly, the theory of (tame) spatial graphs can be studied only in the PL category \cite{Kauffman1989}. 
Henceforth, we assume all bonded knots are PL, though we will often depict them smoothly for aesthetic purposes.

Let us denote the \emph{convex hull} of points $A_1, A_2, \ldots, A_n$ by $[A_1, A_2, \ldots, A_n]$. PL isotopy is generated by the $\Delta$-move (and its inverse) depicted in \autoref{fig:delta_move} \cite{Kauffman1989}, which is a transformation that replaces a segment $[A_i,A_{i+1}]$ of the PL-decomposition of a knot by two segments, $[A_i, A'_{i}] \cup [A'_{i}, A_{i+1}]$, provided that the triangle $[A_i A'_{i} A_{i+1}]$ does not intersect any other segment of the knot (other than $[A_i, A_{i+1}]$). In short:
$$\Delta_{A_i,A_{i+1}}^{A'_i}: K \rightarrow K \setminus [A_i, A_{i+1}] \cup [A_i, A'_{i}] \cup [A_i',A_{i+1}] \qquad \text{given that} \qquad 
K \cap [A_i, A'_i, A_{i+1}] = [A_i, A_{i+1}].
$$

\begin{figure}[ht]
    \centering
    \includegraphics[width=0.3\linewidth,page=25]{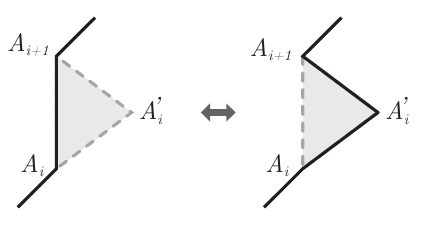}
    \caption{The delta move.}
    \label{fig:delta_move}
\end{figure}

Since bonded knots are special cases of (colored) spatial graphs, we can summarize how isotopy is translated into $\Delta$-moves in the following lemma:

\begin{lemma}[Corollary of Theorem 2.1 in \cite{Kauffman1989}]
Isotopy of topological bonded knots is generated by a finite sequence of $\Delta$-moves represented in \autoref{fig:deltas}.

\end{lemma}

\begin{figure}[ht]
    \centering
    \begin{subfigure}[b]{0.3\textwidth}
        \includegraphics[width=\textwidth, page=26]{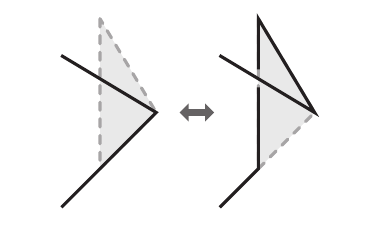}
        \caption{$\Delta_1$}
    \end{subfigure}
    \begin{subfigure}[b]{0.3\textwidth}
        \includegraphics[width=\textwidth, page=27]{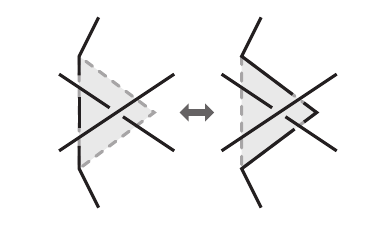}
        \caption{$\Delta_2$}
    \end{subfigure}
    \begin{subfigure}[b]{0.3\textwidth}
        \includegraphics[width=\textwidth, page=28]{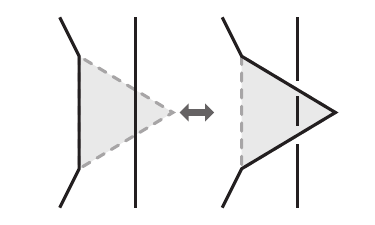}
        \caption{$\Delta_3$}
    \end{subfigure}
    
    \vspace{0.5cm} 
    
    \begin{subfigure}[b]{0.3\textwidth}
        \includegraphics[width=\textwidth, page=29]{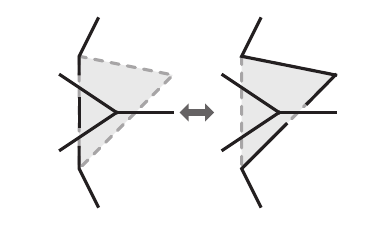}
        \caption{$\Delta_4$}
    \end{subfigure}
    \begin{subfigure}[b]{0.3\textwidth}
        \includegraphics[width=\textwidth, page=30]{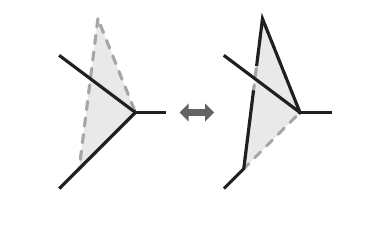}
        \caption{$\Delta_5$}
    \end{subfigure}
    
    \caption{Reidemeister moves of topological bonded knots as $\Delta$-moves in the PL category.}
    \label{fig:deltas}
\end{figure}

If we choose an orientation on the knot $K$, we call such a diagram an \emph{oriented bonded knot diagram}. Note that such a diagram does not specify orientations on the bonds and we will assume they are unoriented, or are given any orientation when necessary.

If, in a given diagram $D$ of a bonded knot $K$, the orientations of the strands adjacent to a bond point in the same direction, we call the bond \emph{parallel}; otherwise, it is \emph{non-parallel} (\autoref{fig:parallel}).  

\begin{figure}[ht]
    \centering
    \includegraphics[width=0.3\linewidth, page=18]{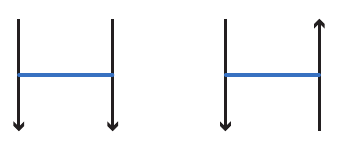}
    \caption{A parallel bond (left) and a non-parallel bond (right).}
    \label{fig:parallel}
\end{figure}

Since the move V reverses the parallelity of a bond, we have the following lemma:
\begin{lemma}
A diagram of an oriented topological bonded knot is equivalent to a diagram of an oriented topological bonded knot a with only parallel bonds.
\end{lemma}

Due to the above lemma, we may always assume, in the topological setting, that a diagram contains only parallel bonds.


We call a piecewise linear diagram $D$ of a bonded knot $K$ \emph{braided around a point $x \notin D$} if all segments $[A_i, A_{i+1}]$ of $D$ are oriented anticlockwise with respect to $x$, i.e. $\angle(A_i,x,A_{i+1}) \in (0, \pi)$ and all bond segments $[B_i,B_{i+1}]$ are colinear with $x$, i.e. $[B_i,B_{i+1}]$ lies on a ray from $x$.

If a diagram $D$ is braided around a point $x$, we can cut $D$ along a ray $r$ from $x$ and open the diagram and fold it into a bonded braid. By a general position argument, we can always choose $r$ that it does not intersect crossings or bonds of $D$. 

\begin{figure}
    \centering
    \includegraphics[width=0.9\linewidth, page=19]{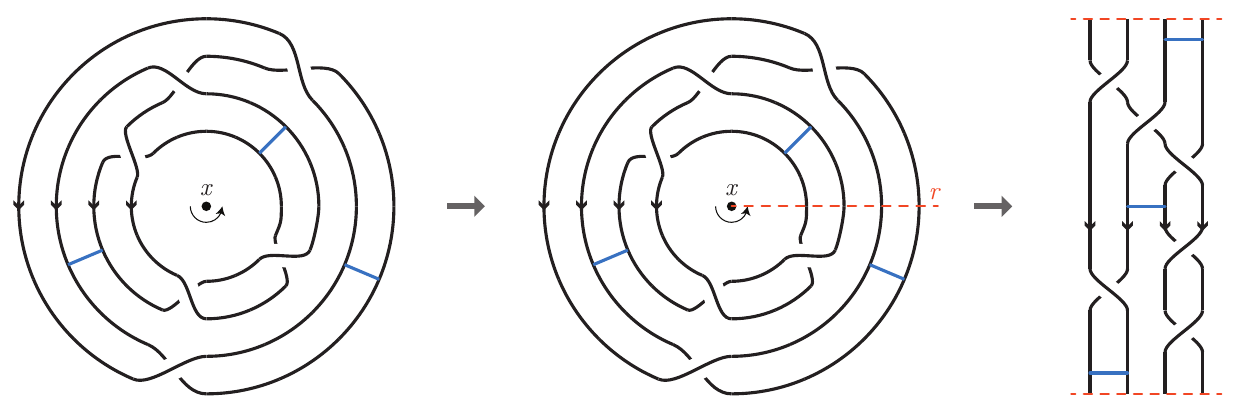}
    \caption{Converting a braided diagram (left) into a braid (right) by cutting the diagram a long a ray (middle). The bonded knot on the left is the closure of the braid on the right, $\beta = b_3 \sigma_1 \sigma_2 \sigma_3^{-1} b_2 \sigma_3 \sigma_1^{-1} \sigma_3 b_1.$}
    \label{fig:braided}
\end{figure}

\begin{theorem}[Alexander's theorem for the topological case]
Every topological bonded knot $K \subset S^3$ can be represented as the closure $\hat \beta$ of an element $\beta$ of the bonded braid monoid $M_n$. \label{thm:topological_akexander}
\end{theorem}

\begin{proof}

Let $K$ be piece-wise linear bonded knot in $S^3$ and let $D$ be a diagram of $K$. 

We showed that in a diagram, we can always isolate a (parallel) bond, from where we can move its endpoints close together via planar isotopy, so that we can represent it by a single segment $[B_i, B'_i]$ (i.e. there exists a sequence of $\Delta$-moves that reduces the number of segments of the bond to a single segment).
Next, we rotate the bond inside a small disk $d$, as presented in \autoref{fig:rotate}, so that it is in braiding form (colinear with $x$) along with its adjacent segments (with possible subdivisions).

\begin{figure}
    \centering
    \includegraphics[width=0.4\linewidth, page=20]{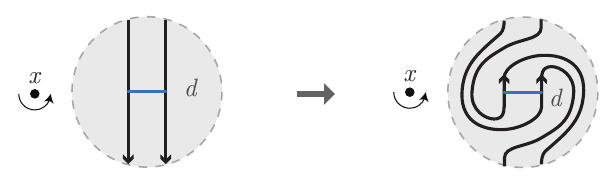}
    \caption{Rotating a bond inside a small disk $d$, so it is braided with respect to the braiding point $x$.}
    \label{fig:rotate}
\end{figure}

Finally, we take each segment $[A_i, A_{i+1}]$ that is not braided, i.e. it is not anti-clockwise with respect to $x$ and make a delta move $\Delta_{A_i,A_{i+1}}^{A_i'}: [A_i, A_{i+1}] \rightarrow [A_i, A_i', A_{i+1}]$, so that the open triangle $\mathrm{Int\,}\triangle([A_i, A_i', A_{i+1}])$ contains $x$, on the diagram this can be achieved so that the arcs $[A_i, A_i']$ and $[A_i', A_{i+1}]$ both contain only over-passes or only under-passes. By a general position argument, the point $A_i'$ can be always chosen in such a way that it does not intersect any crossing, vertex, or bond (in the case it intersects a bond, we can use planar isotopy to make the bond arbitrarily small).

\begin{figure}[ht]
    \centering
    \includegraphics[width = .4\textwidth, page=21]{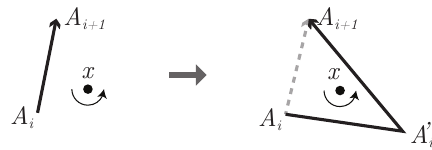}
    \caption{A delta move containing the braiding point $x$ maps a non-braided segment to two braided segments.}
    \label{fig:delta_braid}
\end{figure}

After iterating this process for all non-braided segments, we are left with a braided bonded knot from where we can read the corresponding braid word. 
\end{proof}

We now argue that the set of relations given in Definition \ref{def:mn} suffice to generate bonded braid isotopy.
\begin{proposition}\label{prop:relations}
    The relations in Definition \ref{def:mn} define bonded braid isotopy.
\end{proposition}
\begin{proof}
Let our bonded braid be defined by a braid diagram. Let $\beta$ and $\beta'$ be the bonded braid diagrams which represent the same element in $M_n$ and let $\{\beta_s \mid s \in I\}$ be a family of bonded braid diagrams that join $\beta$ and $\beta'$. 

We subdivide the interval $I = (0,1)$ into finitely many subintervals $(s_i, s_{i+1})$ such that, for each $i$, the family $\{ \beta_s \mid s \in (s_i, s_{i+1}) \}$ undergoes only regular braid isotopy without singularities in the diagram and do not contribute any relations. Transitions between adjacent subintervals are allowed to exhibit one of the following singular events:
\begin{enumerate}
    \item[(1)] two bonds or crossings exchanging relative height,
    \item[(2)] a strand passes over or under a bond or a crossing,
    \item[(3)] the local creation or cancellation of crossings.
\end{enumerate}

All possible cases of (1) are described by relations (R2), (B1), (M1), and (M2) in \autoref{fig:braid-relations}. All possible cases of (2) is described by (R3), (M3) and (M4) in \autoref{fig:braid-relations}. Creation or cancellation of crossings (3) appear in pairs and are  described by $\sigma_i \sigma^{-1} = \sigma^{-1} \sigma_i = 1$ (R1). 
\end{proof}
\begin{theorem}[Alexander's theorem for the rigid case]
Every rigid bonded knot $K \in S^3$ can be represented as the closure $\hat \beta$ of an element $\beta$ of the Rigid Bonded braid monoid $RM_n$. \label{thm:rigid_akexander}
\end{theorem}
\begin{proof} 
A similar braiding technique as in the proof of \autoref{thm:topological_akexander} can be applied to rigid bonded knots, with the additional challenge of handling regions containing kinks. This  can be resolved with the help of Lemma 2.1.2 from \cite{birman1974braids}, which ensures that during the braiding process, there exists a sequence of $\Delta$-moves that transforms any negative segment into positive ones while avoiding any convex regions of the diagram -- in our case, regions containing kinks.




We note here that in the set of generators for the rigid bonded braid group, $RB_n$, we do not need the negative kink, since we can use the braiding algorithm to convert a positive bonded kink to a negative one using the following sequence of Reidemesiter II moves and $\Delta$ moves:
\begin{center}
\includegraphics[scale=0.75, page=37]{braid.pdf}    
\end{center}
In addition, if the bond kink appears on the left-hand side of the diagram, we can move it to the right side via moves RV and II:
\begin{center}
\includegraphics[scale=0.75, page=13]{reid.pdf}    
\end{center}
\end{proof}
The same construction in the proof of Proposition \ref{prop:relations} shows that the relations are defining.

\section{A Burau representation for the bonded and rigid bonded monoid}

The \emph{standard Burau representation} is a linear representation of the braid group $B_n$, given by a homomorphism
$$\psi_n: B_n \to \mathrm{GL}_{n}(\mathbb{Z}[t,t^{-1}]),$$
which arises from the action of $B_n$ on the first homology group of the infinite cyclic cover of the $n$-punctured disk.
Explicitly, the standard Burau representation is given by mapping \cite{burau1935zopfgruppen}
\begin{equation}
\sigma_i \mapsto A_i=
\left(
\begin{array}{c:cc:c}
I_{i-1} & 0 & 0 & 0 \\
\hdashline
0 & 1 - t & t & 0 \\
0 & 1     & 0 & 0 \\
\hdashline
0 & 0     & 0 & I_{n-i-1}
\end{array}
\right). \label{map:burau}
\end{equation}
The goal of this section is to construct a linear representation 
$$\psi: M_n \to \mathrm{GL}_{n}(\mathbb{Z}[t,t^{-1}, z]),$$
of bonded braid monoid.
Since $B_n$ is a subgroup of  $M_n$, we still assume the map \eqref{map:burau},
$$\psi(\sigma_i) = A_i \in \mathrm{GL}_{n}(\mathbb{Z}[t,t^{-1}, z]),$$
for $\sigma_i$ generators.
In order to satisfy relations (B1) and (M1), we assume the representation of $b_i$ is of the form
$$
b_i \mapsto B_i=
\left(
\begin{array}{c:cc:c}
I_{i-1} & 0 & 0 & 0 \\
\hdashline
0 & x & y & 0 \\
0 & z & w & 0 \\
\hdashline
0 & 0     & 0 & I_{n-i-1}
\end{array}
\right). \label{map:burau2}
$$
At this stage, the representation is required to satisfy only (M2), (M3), and (M4), i.e.
$$
\begin{array}{rcl}
\psi (\sigma_i b_i) & = & \psi ( b_i \sigma_i) \\
\psi (\sigma_{i+1} \sigma_i b_{i+1}) & = & \psi(b_i \sigma_{i+1} \sigma_i) \\
\psi(\sigma_{i} \sigma_{i+1} b_{i}) & = & \psi(b_{i+1}\sigma_{i} \sigma_{i+1}) \\
\end{array}
\quad \Longrightarrow \quad
\begin{array}{rcl}
A_i B_i & = & B_iA_i \\
A_{i+1}A_i B_{i+1} & = & B_i A_{i+1}A_i \\
A_i A_{i+1} B_i & = & B_{i+1} A_i A_{i+1} \\
\end{array}
$$
We obtain the following system of equations:
\begin{align*}
A_i B_i - B_iA_i =
\left(
\begin{array}{c:cc:c}
0 & 0 & 0 & 0 \\
\hdashline
0 & -y + t z & t (w - x - y) + y & 0 \\
0 & -w + x + (-1 + t) z & y - t z & 0 \\
\hdashline
0 & 0     & 0 & 0
\end{array}
\right)
&= 0,\\
A_{i+1}A_i B_{i+1} - B_i A_{i+1}A_i =
(t-1)
\left(
\begin{array}{c:cc:c}
0 & 0 & 0 & 0 \\
\hdashline
0 & -1 + x + y & 0 & 0 \\
0 & -1 + w + z & 0 & 0 \\
\hdashline
0 & 0     & 0 & 0
\end{array}
\right)
&= 0,\\
A_i A_{i+1} B_i - B_{i+1} A_i A_{i+1} =
(1-t)
\left(
\begin{array}{c:cc:c}
0 & 0 & 0 & 0 \\
\hdashline
0 & -1 + x + t z & t (-1 + w) + y & 0 \\
0 & 0 & 0 & 0 \\
\hdashline
0 & 0     & 0 & 0
\end{array}
\right)
&= 0,
\end{align*}


The system is linearly underdetermined with one degree of freedom, one possible solution is:
$$x=1-tz, \quad y=tz, \quad z=z, w = 1-z.$$
The \emph{Burau representation of $BB_n$} is thus given by the maps

\begin{align*}
\sigma_i &\mapsto A_i =
\left(
\begin{array}{c:cc:c}
I_{i-1} & 0     & 0 & 0 \\
\hdashline
0       & 1 - t & t & 0 \\
0       & 1     & 0 & 0 \\
\hdashline
0       & 0     & 0 & I_{n-i-1}
\end{array}
\right), \\[1em]
b_i &\mapsto B_i =
\left(
\begin{array}{c:cc:c}
I_{i-1} & 0         & 0     & 0 \\
\hdashline
0       & 1 - tz    & tz    & 0 \\
0       & z         & 1 - z & 0 \\
\hdashline
0       & 0         & 0     & I_{n-i-1}
\end{array}
\right).
\end{align*}
Since the relations for $b_i$ and $k_i$ are independent, we can define a Burau representation for $RB_n$,
$$\psi_R: RB_n \to \mathrm{GL}_{n}(\mathbb{Z}[t,t^{-1}, z, \check z])$$
by:
\begin{align*}
\sigma_i &\mapsto A_i =
\left(
\begin{array}{c:cc:c}
I_{i-1} & 0     & 0 & 0 \\
\hdashline
0       & 1 - t & t & 0 \\
0       & 1     & 0 & 0 \\
\hdashline
0       & 0     & 0 & I_{n-i-1}
\end{array}
\right), \\[1em]
b_i &\mapsto B_i =
\left(
\begin{array}{c:cc:c}
I_{i-1} & 0         & 0     & 0 \\
\hdashline
0       & 1 - tz    & tz    & 0 \\
0       & z         & 1 - z & 0 \\
\hdashline
0       & 0         & 0     & I_{n-i-1}
\end{array}
\right), \\[1em]
k_i &\mapsto C_i =
\left(
\begin{array}{c:cc:c}
I_{i-1} & 0         & 0     & 0 \\
\hdashline
0       & 1 - t\check z    & t\check z    & 0 \\
0       & \check z         & 1 - \check z & 0 \\
\hdashline
0       & 0         & 0     & I_{n-i-1}
\end{array}
\right).
\end{align*}

The Burau representations of the monoids extend naturally to the groups $BB_n$ and $RB_n$.
Since

$$A^{-1}_i =
\left(
\begin{array}{c:cc:c}
I_{i-1} & 0         & 0     & 0 \\
\hdashline
0       & 0    & 1    & 0 \\
0       & t^{-1}         & 1 - t^{-1} & 0 \\
\hdashline
0       & 0         & 0     & I_{n-i-1}
\end{array}
\right)$$
and
$$B^{-1}_i = \frac{1}{1-z-tz}
\left(
\begin{array}{c:cc:c}
I_{i-1} & 0         & 0     & 0 \\
\hdashline
0       & 1-z    & -tz    & 0 \\
0       & -z         & 1-tz & 0 \\
\hdashline
0       & 0         & 0     & I_{n-i-1}
\end{array}
\right),$$
there exists a Burau representation to the general linear group over the ring $R$, where $1-z-tz$ is invertible, i.e. $R=\mathbb{Z}[t, t^{-1}, z, (1-z-tz)^{-1}]$. Explicitly, there exists a linear representation
$$\psi': BB_n \to \mathrm{GL}_{n}(\mathbb{Z}[t,t^{-1}, z, (1-z-tz)^{-1}]),$$
defined by the same matrices as $\psi$.

By the same argument, there exists a linear representation
$$\psi_R': RB_n \to \mathrm{GL}_{n}(\mathbb{Z}[t,t^{-1}, z, \check z, (1-z-tz)^{-1},(1-\hat z-t \hat z)^{-1}]),$$
defined by the same maps as $\psi_R$.


\subsection{The reduced bonded Burau representation}

We now turn to the construction of the reduced bonded Burau representation, which refines the previously defined bonded Burau representation by factoring out the trivial one-dimensional  subrepresentation, analogous to the classical reduction introduced by Burau. The techniques in the original construction of the reduced Burau representation \cite{burau1935zopfgruppen} applies naturally in the bonded setting as well.

\begin{theorem}
Let $n \geq 3$, and let $\sigma_i, b_i \in M_n$. The reduced bonded Burau representation is given explicitly by the mapping
for the braid generators:
\begin{align}
\sigma_1 &
\mapsto A_1' =
\begin{pmatrix}
-t & 0 & 0 \\
1 & 1 & 0 \\
0 & 0 & I_{n-3}
\end{pmatrix}, \qquad
\sigma_{n-1}
\mapsto A_{n-1}' =
\begin{pmatrix}
I_{n-3} & 0 & 0 \\
0 & 1 & t \\
0 & 0 & -t
\end{pmatrix},
\notag
\\
\sigma_{i} & \mapsto 
A_i' =
\begin{pmatrix}
I_{i-2} & 0 & 0 & 0 & 0 \\
0 & 1 & t & 0 & 0 \\
0 & 0 & -t & 0 & 0 \\
0 & 0 & 1 & 1 & 0 \\
0 & 0 & 0 & 0 & I_{n-i-2}
\end{pmatrix}\quad \text{for \; $1 < i < n-1$}. \notag
\end{align} 

and for the bond generators:

\begin{align}
b_1 &
\mapsto B_1' =
\begin{pmatrix}
-tz - z + 1 & 0 & 0 \\
z & 1 & 0 \\
0 & 0 & I_{n-3}
\end{pmatrix}, \qquad
b_{n-1}
\mapsto B_{n-1}' =
\begin{pmatrix}
I_{n-3} & 0 & 0 \\
0 & 1 & tz \\
0 & 0 & -tz - z + 1
\end{pmatrix},
\notag
\\
b_{i} &\mapsto 
B_i' =
\begin{pmatrix}
I_{i-2} & 0 & 0 & 0 & 0 \\
0 & 1 & tz & 0 & 0 \\
0 & 0 & -tz - z + 1 & 0 & 0 \\
0 & 0 & z & 1 & 0 \\
0 & 0 & 0 & 0 & I_{n-i-2}
\end{pmatrix}\quad \text{for \; $1 < i < n-1$}. \notag
\end{align}


\end{theorem}

\begin{proof}
Let us first show how the reduced presentation is obtained. For all $i = 1, \dots, n-1$ we have:
$$
C^{-1} A_i C =
\begin{pmatrix}
A_i' & 0 \\
\ast_i & 1
\end{pmatrix},
$$
where $C = C_n$ is the $n \times n$ matrix
$$
C =
\begin{pmatrix}
1 & 1 & \cdots & 1 \\
0 & 1 & \cdots & 1 \\
0 & 0 & \ddots & \vdots \\
0 & 0 & \cdots & 1
\end{pmatrix},
$$
and $\ast_i$ is the row vector in $\mathbb{Z}^{n-1}$ given by:
$$
\ast_i =
\begin{cases}
(0, \dots, 0), & i < n-1 \\
(0, \dots, 0, 1), & i = n-1.
\end{cases}
$$
Similarly, for all $i = 1, \dots, n-1$ we have:
$$
C^{-1} B_i C =
\begin{pmatrix}
B_i' & 0 \\
\ast_i & 1
\end{pmatrix},
$$
with $\ast_i$ as above, but for $i = n-1$ given by $(0, \dots, 0, z)$. For each $i = 1, \dots, n-1$, define:
$$
A_i'' :=
\begin{pmatrix}
A_i' & 0 \\
\ast_i & 1
\end{pmatrix},
\quad
B_i'' :=
\begin{pmatrix}
B_i' & 0 \\
\ast_i & 1
\end{pmatrix}.
$$
In order to prove the theorem, it suffices to show that:
$$
A_i C = C A_i'', \quad \text{and} \quad B_i C = C B_i''.
$$
We observe that both $A_i C$ and $C A_i''$ are obtained from $C$ by the same row operations:
\begin{itemize}
    \item replace the entry in row $i$, column $i$, with $1 - t$,
    \item replace the entry in row $i+1$, column $i$, with $1$.
\end{itemize}
Similarly, $B_i C$ and $C B_i''$ are obtained from $C$ by:
\begin{itemize}
    \item replacing entry $(i,i)$ with $1 - tz$,
    \item replacing entry $(i+1,i)$ with $z$.
\end{itemize}

Thus, $A_i C = C A_i''$ and $B_i C = C B_i''$, proving the conjugation formulas.\\
\\
Since the matrices $A_1, \dots, A_{n-1}$ and $B_1, \dots, B_{n-1}$ satisfy the braid relations (with bonded extensions), their conjugates under $C^{-1} (-) C$ also satisfy these relations.

From the decomposition:
$$
C^{-1} A_i C = \begin{pmatrix} A_i' & 0 \\ \ast_i & 1 \end{pmatrix}, \quad
C^{-1} B_i C = \begin{pmatrix} B_i' & 0 \\ \ast_i & 1 \end{pmatrix},
$$
we see that the matrices $A_i'$ and $B_i'$ also satisfy braid-type relations.

Thus, we obtain reduced representation:
$$
\psi_n^r: M_n \longrightarrow \mathrm{GL}_{n-1}(\Lambda), \quad \sigma_i \mapsto A_i', \quad b_i \mapsto B_i',
$$
which we call the \emph{reduced bonded Burau representation}.

For $n = 2$, these reduce to:
$$
\psi_2^r : M_2 \longrightarrow \mathrm{GL}_1(\Lambda), \quad \sigma_1 \mapsto -t, \quad b_1 \mapsto -tz - z + 1.
$$

\end{proof}

\begin{corollary}
The (bonded) Burau representation is reducible for all $n \geq 2$.
\end{corollary}

\begin{proof}
Let $\Lambda = \mathbb{Z}[t^{\pm 1}, z]$ and consider the topological Burau representations:
$$
\psi_n: B_n \longrightarrow \mathrm{GL}_n(\Lambda), \quad \sigma_i \mapsto A_i, \quad b_i \mapsto B_i.
$$

\noindent \textbf{Case 1: $n \geq 3$.} Let $e_1, \dots, e_n$ be the standard basis of $\Lambda^n$, and consider the $\Lambda$-submodule $\langle C e_n \rangle$, where $C = C_n$ is the upper-triangular matrix with all ones on and above the diagonal:
$$
C =
\begin{pmatrix}
1 & 1 & \cdots & 1 \\
0 & 1 & \cdots & 1 \\
0 & 0 & \ddots & \vdots \\
0 & 0 & \cdots & 1
\end{pmatrix}.
$$
Using the previous theorem, for both braid and bond generators we have:
$$
C^{-1} A_i C = \begin{pmatrix} A_i' & 0 \\ \ast_i & 1 \end{pmatrix}, \quad
C^{-1} B_i C = \begin{pmatrix} B_i' & 0 \\ \ast_i & 1 \end{pmatrix}.
$$
Then:
$$
A_i C e_n = C \begin{pmatrix} A_i' & 0 \\ \ast_i & 1 \end{pmatrix} e_n = C e_n,
\quad
B_i C e_n = C \begin{pmatrix} B_i' & 0 \\ \ast_i & 1 \end{pmatrix} e_n = C e_n.
$$
So the submodule $\langle C e_n \rangle$ is invariant under both the braid and bond parts of the representation, implying reducibility.

\medskip

\noindent\textbf{Case 2: $n = 2$.} 
Let
$$
U_1 = \begin{pmatrix} 1 - t & t \\ 1 & 0 \end{pmatrix}, \quad
B_1 = \begin{pmatrix} 1 - tz & tz \\ z & 1 - z \end{pmatrix}, \quad
C = \begin{pmatrix} 1 & 1 \\ 0 & 1 \end{pmatrix}.
$$
We have:
$$
C^{-1} U_1 C = \begin{pmatrix} 1 & -1 \\ 0 & 1 \end{pmatrix}
\begin{pmatrix} 1 - t & t \\ 1 & 0 \end{pmatrix}
\begin{pmatrix} 1 & 1 \\ 0 & 1 \end{pmatrix}
= \begin{pmatrix} -t & 0 \\ 1 & 1 \end{pmatrix},
$$

$$
C^{-1} B_1 C = \begin{pmatrix} 1 & -1 \\ 0 & 1 \end{pmatrix}
\begin{pmatrix} 1 - tz & tz \\ z & 1 - z \end{pmatrix}
\begin{pmatrix} 1 & 1 \\ 0 & 1 \end{pmatrix}
= \begin{pmatrix} -tz - z + 1 & 0 \\ z & 1 \end{pmatrix}.
$$
Now consider the submodule $\langle C e_2 \rangle$. Then:
$$
U_1 C e_2 = C \begin{pmatrix} -t & 0 \\ 1 & 1 \end{pmatrix} e_2 = C e_2,
\quad
B_1 C e_2 = C \begin{pmatrix} -tz - z + 1 & 0 \\ z & 1 \end{pmatrix} e_2 = C e_2.
$$
Thus, the submodule $\langle C e_2 \rangle$ is invariant in both the braid and bond cases for $n = 2$ as well.

\medskip

In all cases $n \geq 2$, we have a proper, nontrivial $\Lambda$-submodule preserved by the (bonded) Burau representation. Hence, the representation is reducible.
\end{proof}

We now turn to the question of faithfulness for the reduced bonded Burau representation.

\medskip

\noindent For $n=1$ it is obvious that it is faithful. 

\begin{proposition}
The reduced bonded Burau representation is faithful for $ n = 2 $; that is, the homomorphism
$$
\psi_2^r : M_2 \longrightarrow \mathrm{GL}_1(\mathbb{Z}[t^{\pm1}, z])
$$
given by
$$
\sigma_1 \mapsto -t, \quad b_1 \mapsto 1 - z - t z
$$
is injective.
\end{proposition}

\begin{proof}
Let $ w \in M_2 $ be any word in the generators $ \sigma_1^{\pm 1}, b_1 $, and suppose $\psi_2^r(w) = 1 \in \mathbb{Z}[t^{\pm1}, z]$.
Since
$$
\psi_2^r(\sigma_1) = -t \quad \text{(a unit)}, \quad
\psi_2^r(b_1) = 1 - z - t z \quad \text{(not a unit)},
$$
any word $ w $ that maps to $ 1 $ must contain no $ b_1 $ letters (since $ 1 - z - t z $ is not invertible), and
therefore be a power of $ \sigma_1 $, i.e. $ w = \sigma_1^k,\; k\in\mathbb{Z} $

But:
$$
\psi_2^r(\sigma_1^k) = (-t)^k = 1 \iff k = 0
$$
in $ \mathbb{Z}[t^{\pm1}] $. Therefore, $ w $ must be the trivial word.

Hence, the representation is injective.
\end{proof}




It was shown in \cite{long1993burau} and \cite{bigelow1999burau} that the Burau representation of the Artin braid group is not faithful for $ n = 5 $ and for all $ n \geq 6 $, respectively. Therefore, the bonded Burau representation is also not faithful for $ n \geq 5 $. 
For $n=3$, it was shown in \cite{dasbach2000faithful} that the representation for the singular braid monomial $SB_n$ \cite{gemein1997singular} is faithfull for three strands. Since $SB_n$ is isomorphic to $M_n$ and the representation defined in \cite{dasbach2000faithful} for $SB_n$ is isomorphic to ours, it follows that reduced bonded Burau representation is faithful for $n = 3$.  The faithfulness for $n = 4$, as in the classical case, remains unknown.

    



\section{Markov theorems}

A fundamental result in braid theory is the Markov theorem, which provides necessary and sufficient conditions for two braids to represent the same link -- that is, it characterizes when their closures yield equivalent links. We begin by recalling the classical version of this theorem.

\begin{theorem}[\cite{markoff1936uuber, birman1974braids}]
    Let $\beta_1 \in B_n$ and $\beta_2 \in B_m$ be two bonded braids.
    Their closures $\hat{\beta_1}$ and $\hat{\beta_2}$  are equivalent knots if and only if  $\beta_1$ and $\beta_2$ are related by a finite sequence of the following operations:
    \begin{enumerate}
        \item conjugation: $\beta \rightarrow \sigma_i^{\pm 1} \beta \sigma_i^{\mp 1}$, where $\beta \in B_n$ and $1 \leq i < n$,
        \item stabilization: $\beta \leftrightarrow \beta \sigma_n^{\pm 1}$, where $\beta \in B_n$ and $\beta \sigma_n^{\pm 1} \in B_{n+1}$. 
    \end{enumerate}
\end{theorem}

We call two braids $\beta \in B_n$ and $\beta' \in B_{n'}$ \emph{Markov equivalent} if they are connected through braid isotopy, conjugation and stabilization.

The Markov theorem was first formulated in \cite{markoff1936uuber}, where a sketch of the proof was presented. A complete and rigorous proof was later given by Birman \cite{birman1974braids}. While Birman’s approach is thorough, it relies on manipulating the combinatorics of $\Delta$-moves. Since then, more geometrically motivated and arguably more accessible proofs have been developed by Morton \cite{morton1986threading}, Yamada \cite{yamada1987minimal}, and Vogel \cite{vogel1990representation}.

\begin{theorem}[The Markov theorem for topological bonded braids]
    Let $\beta_1 \in B_n$ and $\beta_2 \in B_m$ be two bonded braids.
    Their closures $\hat{\beta_1}$ and $\hat{\beta_2}$  are equivalent knots if and only if  $\beta_1$ and $\beta_2$ are related by a finite sequence of the following operations:
    \begin{enumerate}
    \item conjugation: $\beta \rightarrow \sigma_i^{\pm 1} \beta \sigma_i^{\mp 1}$, where $\beta \in M_n$ and $1 \leq i < n$,
    \item cyclic permutation of bonds: $\beta b_i \leftrightarrow b_i \beta$, where $\beta \in M_n$ and $1 \leq i < n$,
    \item stabilization: $\beta \leftrightarrow \beta \sigma_n^{\pm 1}$, where $\beta \in M_n$ and $\beta \sigma_n^{\pm 1} \in M_{n+1}$. 
    \end{enumerate}
\end{theorem}

We call two braids $\beta \in M_n$ and $\beta' \in M_{n'}$ \emph{bonded Markov equivalent} if they are connected through braid isotopy, conjugation, cyclic permutations of bonds, and stabilization.

\begin{proof}
Our approach will follow that of Burde-Zieschang \cite{burde2002knots}, where Morton’s proof \cite{morton1986threading} is presented in a more modern form. We will first summarize Morton's proof using notation from \cite{burde2002knots} for the classical case and then adapt it to topological bonded braids.

The braiding procedure begins with a diagram $D$ of an oriented knot $K$, and proceeds by selecting two finite, alternating subsets of \emph{intersection points} on the diagram which must not lie on a crossing, $S=\{s_1,s_2,\ldots,s_n\}$ and
$F=\{f_1,f_2,\ldots,f_n\}$. The points in $S$ and $F$ mark \emph{starting} and \emph{ending} points of \emph{overpasses}, respectively. 
We assume that the points $S \cup F$ appear consecutively along the diagram, following its orientation, in the order $s_1, f_1, s_2, f_2, \ldots, s_n, f_n$ so that every subarc from $s_i$ to $f_i$ do not contain underpasses (they are overpasses), and the segments between $f_i$ and $s_{(i+1) \pmod n}$ do not contain overpasses (they are underpasses). Both types of subarcs may be trivial, i.e.\ they contain no overpasses or underpasses.
Note that consecutive indexing is not strictly necessary, but it simplifies the notation in some of the arguments.

It is easy to see (\cite{burde2002knots}, page 165) that there always exist a simple closed curve $h$ in the plane of $D$ with the following conditions: 
\begin{itemize}
    \item $h$ meets $D$ transversely and does not intersect crossings or intersection points,
    \item $h$ separates the set $S$ and $F$,
    \item $h$ intersects each overpass $s_i f_i$ and each underpass $s_i f_{(i+1) \pmod n}$ exactly once,
    \item $h$ crosses open overpasses of $D$ from left to right and open underpasses of $D$ from right to left as depicted in \autoref{fig:leftright}, where we also eqip $h$ with overpasses and underpasses.
\end{itemize}
We call such an $h$ the \emph{threading of $D$} with respect to $S$ and $F$.

\begin{figure}[ht]
    \centering
    \includegraphics[page=45, scale=1.0]{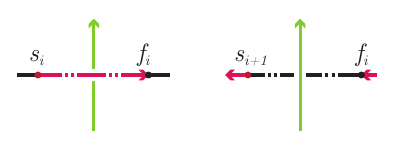}
    \caption{Orientations of crossings between the threading curve $h$ and $D$.}
    \label{fig:leftright}
\end{figure}

\medskip

From here, we can straighten $h \cup K$, so that $h$ forms a line through infinity and we obtain the knot $K$ braided along $h$, i.e. in the cylinder with axis $h$, the polar coordinate of $K$ is monotonically increasing by construction (for more details, see \cite{burde2002knots}, where it is also explained how to obtain the braid word from the threading). A full example, from where the whole braiding process is visible is given in \autoref{fig:threading}. Described above is the Alexander theorem for braids.

\begin{figure}[ht]
    \centering
    \includegraphics[page=39, scale=1.0]{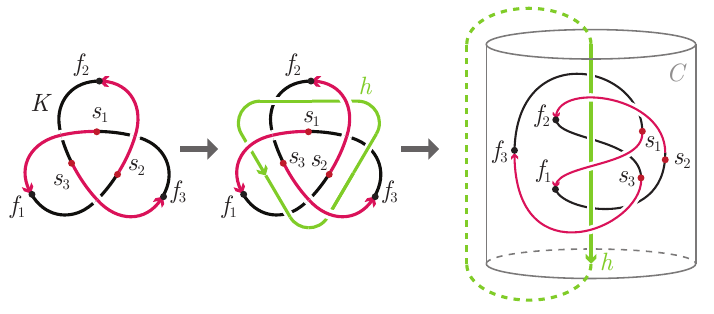}\label{leftright}
    \caption{The threading and braiding process for knots.} 
    \label{fig:threading}
\end{figure}

The reasoning behind the Markov theorem is as follows: if the threading axis $h$ is fixed, then the braid represented by $K \subset \mathbb{R}^3 - h$ (or $K \subset C - h$) is determined uniquely up to braid isotopy and conjugation in the braid group (\cite{morton1986threading}, p. 248, Section 2, complete proof in \cite{morton1978infinitely}).

It remains to study what happens when we keep $K$ fixed and study the consequences of changing $h$.

The threading $h$ may be isotoped over an arc of $K$ by the moves $\Delta'$ depicted in \autoref{fig:change_of_ha}. It may be also isotoped over a crossing of $K$ by moves $\Delta''$ depicted in \autoref{fig:change_of_hb}. The threading $h$ may be also passed through a crossing by move $\Delta'''$ depicted in \autoref{fig:change_of_h3a}. On all diagrams the intersections point markers are omitted.

The $\Delta'$ moves add an extra loop around the threading $h$ and can be expressed as braid isotopy and stabilization (see Figure \autoref{fig:stab}), while the moves $\Delta''$ are isotopies of the closed braid. The move $\Delta'''$ is just a sequence of moves $\Delta'$ and $\Delta''$ (\autoref{fig:change_of_h3b}).

\begin{figure}[ht]
    \centering
    \begin{subfigure}[b]{\textwidth}\centering
        \includegraphics[page=40, scale=1.0]{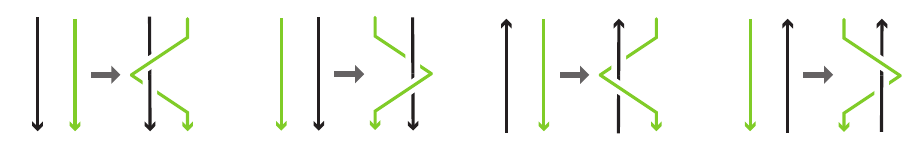}
        \caption{moves $\Delta'$}\label{fig:change_of_ha}
    \end{subfigure}
    \\
    \centering
    \begin{subfigure}[b]{\textwidth}\centering
        \includegraphics[page=49, scale=1.0]{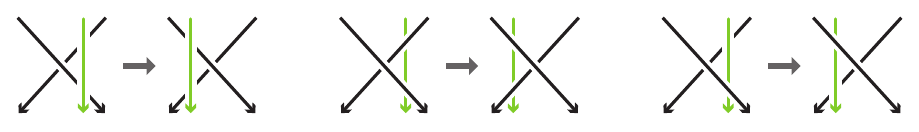}
        \caption{moves $\Delta''$}\label{fig:change_of_hb}
    \end{subfigure}
    \caption{Moves on $h$ when it passes a strand of $K$ (a) or a crossing of $K$ (b).} 
    \label{fig:change_of_h}
\end{figure}

\begin{figure}[ht]
    \centering
    \begin{subfigure}[b]{0.3\textwidth}\centering
        \includegraphics[page=50, scale=1.0]{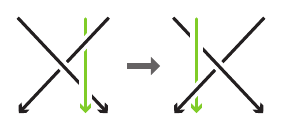}
        \caption{move $\Delta'''$}\label{fig:change_of_h3a}
    \end{subfigure}
    \begin{subfigure}[b]{0.6\textwidth}\centering
        \includegraphics[page=43, scale=1.0]{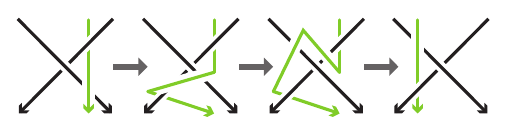}
        \caption{move $\Delta'''$ can be expressed by $\Delta'$ and $\Delta''$ moves}\label{fig:change_of_h3b}
    \end{subfigure}
    \caption{Moves on $h$ when it passes a strand of $K$ (a) or a crossing of $K$ (b).} 
    \label{fig:change_of_h3}
\end{figure}

\begin{figure}[ht]
    \centering
    \begin{subfigure}[b]{0.45\textwidth}\centering
        \includegraphics[page=57, scale=1]{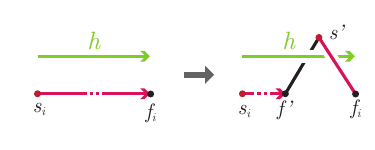}
        \caption{a case of move $\Delta'$}\label{fig:change_of_ha1}
    \end{subfigure}
    \begin{subfigure}[b]{0.45\textwidth}\centering
        \includegraphics[page=58, scale=1]{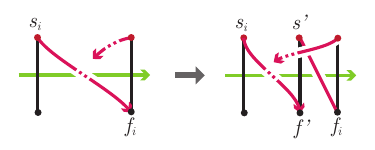}
        \caption{braiding after $\Delta'$}\label{fig:change_of_hb1}
    \end{subfigure}
    \caption{The move $\Delta'$ has the effect of adding a loop around $h$ (a) and is expressed as stabilization after straightening the braid (b).} 
    \label{fig:stab}
\end{figure}

In general, the isotopy that moves $h$ into another axis, $h'$, will not be an isotopy in $\mathbb{R}^2 \setminus (S \cup F)$. Let the axis $h$ be straightened and let $h'$ be another threading that separates $S$ and $F$, see the example in \autoref{fig:pusha}. We can draw parallel horizontal rays $r_i$ from points in $F$ towards infinity away from $h$ and observe that there are always an even number of intersections between $r_i$ and $h'$ and two arcs consecutively intersecting $r_i$ are always oppositely oriented . We can now iterate the process of pushing a pair of oppositely oriented arcs over the points $f_i$ as depicted in \autoref{fig:pushb}. This process of pushing two oppositely oriented arcs is an isotopy of the closed braid. After there are no more intersections with the rays, we obtain a threading $h''$ that is isotopic to $h$ in $\mathbb{R} - (S \cup F)$. 
The last thing to check is that the the Reidemeister moves of $K$ do not interfere with the threading, which is easy to check (see \cite{burde2002knots}, page 169). 
This finishes the proof for the classical Markov theorem.

\begin{figure}[ht]
    \centering
    \begin{subfigure}[b]{0.45\textwidth}\centering
        \includegraphics[page=51, scale=1]{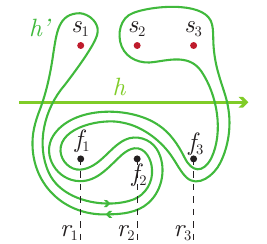}
        \caption{two non-isotopic threadings in $\mathbb{R}^2 - (S \cup F)$}\label{fig:pusha}
    \end{subfigure}
    \begin{subfigure}[b]{0.45\textwidth}\centering
        \includegraphics[page=44, scale=1]{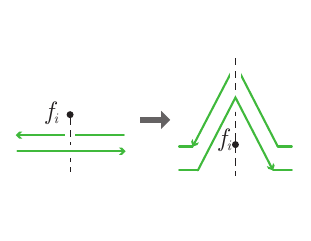}
        \caption{pushing $h'$ over $f_i$}\label{fig:pushb}
    \end{subfigure}
    \caption{Moves on $h$ when it passes a strand of $K$ (a) or a crossing of $K$ (b).} 
    \label{fig:push}
\end{figure}

Let us now consider the case of topological bonded knots. For the threading we proceed almost entirely the same as in the classical case, but we add the following three conditions:
\begin{enumerate}
    \item the intersection points $S$ and $F$ lie on the knotted part of the bonded knot (and not on the bonds),
    \item bond vertices lie entirely on overpasses $s_i f_i$ that do not contain any crossings,
    \item the threading curve $h$ must not intersect the bond. 
\end{enumerate}

A full example of a threading of a bonded knot is presented in \autoref{fig:threada}. The straightening procedure from where we can read the braid word $\beta$ of the bonded knot $\hat \beta$ is presented in \autoref{fig:threadb} and \autoref{fig:threadc}.

\noindent \emph{Remark.} There are several alternative ways we can define the threading rules above (e.g., requiring that bonds lie on underpasses or that the bond vertices must be contained in the set $F$), but we have chosen the version that appears to be the simplest and leads to the clearest arguments.

Let $D$ be now a bonded knot diagram of a bonded knot $K$. A separating curve $h$ following the new rules exists: if $h'$ is a separating curve on the diagram with bonds removed (which exist due to Morton), we can make a local change depicted in \autoref{fig:hh} to $h'$ to obtain a separating curve $h$ of the diagram with bonds attached which follows rules 1, 2, and 3.

\begin{figure}[ht]
\centering
       \includegraphics[page=62, scale=1]{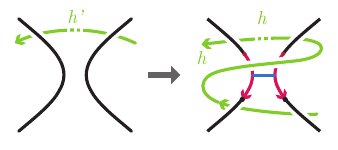}
        \caption{Obtaining a bonded threading $h$ from the threading $h'$ of the diagram with bonds removed.}\label{fig:hh}
\end{figure}

Since in bonds lie only on the overstrands without crossings, they will contain no crossings once we straighten the threading curve.
Straightening $h$ gives rise to a bonded knot braided along the axis $h$ inside a cylinder $C$, which gives an alternative proof of \autoref{thm:topological_akexander}. 


As before, fixing $h$ determines the braid of $K$ up to bonded braid isotopy, conjugation and cyclic permutations of bonds.

\begin{figure}[ht]
    \centering
\centering
    \begin{subfigure}[b]{0.9\textwidth}\centering
        \includegraphics[page=59, scale=1]{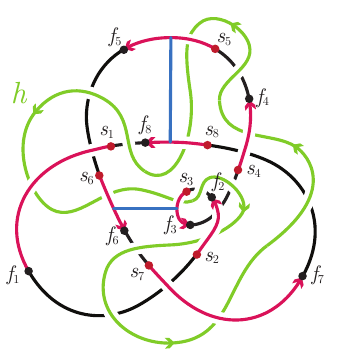}
        \caption{threading $h$ of a bonded knot}\label{fig:threada}
    \end{subfigure}
    \begin{subfigure}[b]{0.45\textwidth}\centering
        \includegraphics[page=60, scale=1]{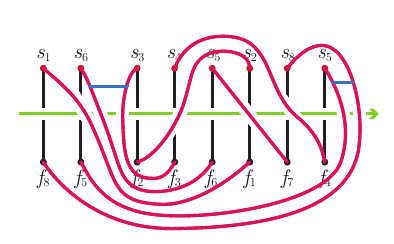}
        \caption{straightening the threading axis $h$}\label{fig:threadb}
    \end{subfigure}
        \begin{subfigure}[b]{0.45\textwidth}\centering
        \includegraphics[page=61, scale=1]{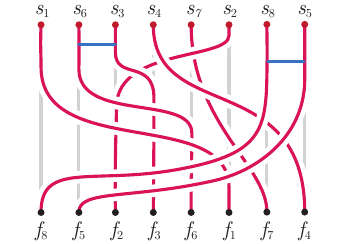}
        \caption{pushing overpasses to obtain the braid}\label{fig:threadc}
    \end{subfigure}
    \caption{The braiding procedure and braid word reconstruction for a topologically bonded knot. The obtained braid word representing the bonded knot is 
    $\beta = b_2 \sigma_5 \sigma_4^{-1} \sigma_3^{-1} b_7 \sigma_4^{-1} \sigma_2^{-1} \sigma_4^{-1} \sigma_1^{-1} \sigma_2^{-1} \sigma_3^{-1} \sigma_6 \sigma_5 \sigma_4 \sigma_3 \sigma2 \sigma_1 \sigma_7 \sigma_6 \sigma_5 \sigma_4 \sigma_3 \sigma_2 \in M_8.$}
    \label{fig:threaddd}
\end{figure}

It again remains to see what happens if we keep $K$ fixed and study the consequences of changing $h$. In addition to the moves $\Delta'$, $\Delta''$, and $\Delta'''$, 
there is one more move $\Delta^{(4)}$, which arises when the threading curve is pushed through a bond. This move, illustrated in \autoref{fig:bond_delta}, corresponds to an isotopy of the closed braid.

\begin{figure}[ht]
    \centering
    \includegraphics[page=53, scale=1]{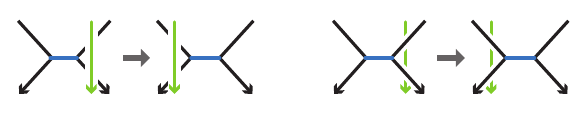}
    \caption{The move $\Delta^{(4)}$ when $h$ passes a bond.}\label{fig:bond_delta}
\end{figure}

The rest of the proof is identical to that in the classical case, except we must study the effect of threading when we perform the move $\mathrm{VI_B}$ (and $\mathrm{VI'_B}$) on $K$.
We will study only $\mathrm{VI_B}$, since $\mathrm{VI'_B}$ is analogous.
As we assumed in \autoref{sec:alexander}, the bond is parallel. In the first case we assume the overpassing arc of $\mathrm{VI'_B}$ is oriented downwards as in \autoref{fig:m4ba}.
Similarly, if the overpassing arc is oriented upwards, we have only two cases presented in \autoref{fig:m4bb}.

The move in \autoref{fig:m4ba} is just braid isotopy, while the move \autoref{fig:m4bb} is braid isotopy and stabilization as shown in \autoref{fig:m4b2}, which concludes the proof.

\begin{figure}[ht]
    \centering
    \begin{subfigure}[b]{0.47\textwidth}
    \includegraphics[page=54, scale=1]{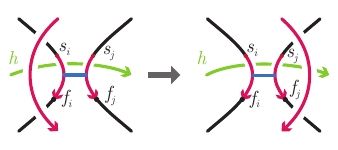}\caption{}\label{fig:m4ba}
    \end{subfigure}\quad
    \begin{subfigure}[b]{0.47\textwidth}
    \includegraphics[page=55, scale=1]{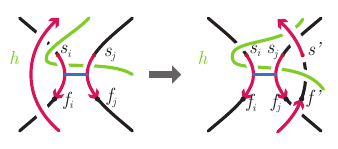}\caption{}\label{fig:m4bb}
    \end{subfigure}
    \caption{The effect of threading under the move $\mathrm{VI_B}$.}
   \label{fig:m4b}
 \end{figure}

\begin{figure}[ht]
    \centering
    \includegraphics[page=56, scale=1]{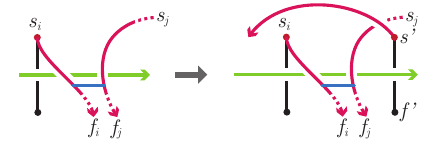}\label{fig:m4ba2}
    \caption{The braid change under the threading change of the move from \autoref{fig:m4bb}.}
    \label{fig:m4b2}
\end{figure}

\end{proof}

\noindent \emph{Remarks.} Using the same methods as above, one can derive the Markov theorem in the rigid setting as well. This version is analogous to the topological case, with the key difference that cyclic permutations of the $k$-generators are also permitted, i.e., the move of the form $\beta k_i \leftrightarrow k_i \beta$. In the case of the bonded braid group, conjugation and stabilization alone suffice to generate Markov equivalence. Although moves involving ``antibonds'' can be formally defined, they do not carry a clear physical interpretation.


\subsection*{Acknowledgments}
B. Gabrovšek was financially supported by the Slovenian Research and Innovation Agency grant N1-0278. M. Simonič and P. Cavicchioli were financially supported by the Slovenian Research and Innovation Agency grant J1-4031  and and program P1-0292. The authors declare that they have no conflict of interest.

\bibliographystyle{abbrv}
\bibliography{biblio}

\end{document}